\documentclass[11pt,a4paper]{article}
\usepackage[top=2.5cm, bottom=2.5cm, left=2.5cm, right=2.5cm]{geometry}
\usepackage[latin1]{inputenc}
\usepackage{csquotes}
\usepackage{amsmath}
\usepackage{amsthm}
\usepackage{amsfonts}
\usepackage{amssymb}
\usepackage{graphics}
\usepackage{float}
\usepackage[hang,flushmargin]{footmisc} %nonindent footnote

\usepackage{amsmath,amsthm,amssymb,graphicx, multicol, array}
\usepackage{enumerate}
\usepackage{enumitem}
\usepackage{xcolor}
\usepackage[pagebackref,bookmarks,colorlinks,breaklinks]{hyperref}
\hypersetup{linkcolor=blue,citecolor=red,filecolor=blue,urlcolor=blue} 

\usepackage{tabto}

\numberwithin{equation}{section}

\usepackage{tikz}
\usepackage{pgfplots}
\pgfplotsset{compat=1.16}
\usetikzlibrary{matrix}
\usepackage{mathrsfs}

\DeclareMathOperator{\ord}{ord}
\DeclareMathOperator{\Ord}{Ord}

\DeclareMathOperator{\tr}{Tr}

\newtheorem{thm}{Theorem}[section]
\newtheorem{lem}{Lemma}[section]
\newtheorem{prop}{Proposition}[section]

\newtheorem{cor}{Corollary}[section]
\newtheorem{dfn}{Definition}[section]

\newcommand{\N}{\mathbb{N}}
\newcommand{\Z}{\mathbb{Z}}

\newcommand{\C}{\mathbb{C}}
\newcommand{\F}{\mathbb{F}}

\usepackage{fancyhdr}
\usepackage{titletoc}
\let\LaTeXStandardTableOfContents\tableofcontents

\renewcommand{\tableofcontents}{%
	\begingroup%
	\renewcommand{\bfseries}{\relax}%
	\LaTeXStandardTableOfContents%
	\endgroup%
}%

\title{Equidistribution of Primitive Normal Elements in Finite Fields}
\date{}
\author{N. A. Carella}

\begin{document}

	%\doublespacing
\thispagestyle{empty}
\date{}
	\maketitle

\vskip .25 in 
\begin{abstract}
Let $q=p^k$ be a prime power, let $\mathbb{F}_q$ be a finite field and let $n\geq2$ be an integer. It is shown that the set of primitive normal elements in a finite field extension $\mathbb{F}_{q^n}$ is a Salem set. Furthermore, it is proved that this set is strongly equidistributed in the finite field $\mathbb{F}_{q^n}$. Similar results are proved for the set of quadratic residues and the set of primitive roots modulo a prime.    \let\thefootnote\relax\footnote{ \today \date{} 
		}
\end{abstract}

\setcounter{tocdepth}{1}
\tableofcontents
%SSSSSSSSSSSSSSSSSSSSSSSSSSSSSSSSSSSSSSSSSSSSSSSSSSSSSSSSSSSSSSSSSSSSSSSSSSSSSSSSSSS
%SSSSSSSSSSSSSSSSSSSSSSSSSSSSSSSSSSSSSSSSSSSSSSSSSSSSSSSSSSSSSSSSSSSSSSSSSSSSSSSSSSS
%SSSSSSSSSSSSSSSSSSSSSSSSSSSSSSSSSSSSSSSSSSSSSSSSSSSSSSSSSSSSSSSSSSSSSSSSSSSSSSSSSSS
%SSSSSSSSSSSSSSSSSSSSSSSSSSSSSSSSSSSSSSSSSSSSSSSSSSSSSSSSSSSSSSSSSSSSSSSSSSSSSSSSSSS
%SSSSSSSSSSSSSSSSSSSSSSSSSSSSSSSSSSSSSSSSSSSSSSSSSSSSSSSSSSSSSSSSSSSSSSSSSSSSSSSSSSS
%SSSSSSSSSSSSSSSSSSSSSSSSSSSSSSSSSSSSSSSSSSSSSSSSSSSSSSSSSSSSSSSSSSSSSSSSSSSSSSSSSSS
%SSSSSSSSSSSSSSSSSSSSSSSSSSSSSSSSSSSSSSSSSSSSSSSSSSSSSSSSSSSSSSSSSSSSSSSSSSSSSSSSSSS
\section{Introduction} \label{S4343PNEFF-I}\hypertarget{S4343PNEFF-I}
Let $q$ be a prime power, let $n\geq2$ be an integer and let $\F_{q^n}$ be a finite field. A normal element $\eta\in \F_{q^n}$ is a generator of the finite field as an additive group 
\begin{equation} \label{eq4343PNEFF.050b}
\F_{q^n}=\{ \alpha=a_0\eta+ a_1\eta^q+a_2\eta^{q^2}+\cdots+a_{n-1}\eta^{q^{n-1}}:a_i\in \F_q\}
\end{equation}	
and a primitive element $\eta\in \F_{q^n}$ generator of the finite field as a multiplicative group 
\begin{equation} \label{eq4343PNEFF.050d}
	\F_{q^n}^{\times}=\{ \alpha=\eta^k:k\in [0,q^n-2]\}.
\end{equation}	
A primitive normal element $\eta\in \F_{q^n}$ is a simultaneous generator of both groups. Various results on the theory of primitive element normal elements in finite fields are investigated in \cite{HK1888}, \cite{OO1934}, \cite{CL1952}, \cite{DH1968}, \cite{LS1987}, \cite{CK2024}, \cite{FR2024}, et alia. The asymptotic formula for the number primitive normal elements in the finite field $\F_{q^n}$, that is, 
\begin{align} \label{eq4343PNEFF.050i10}
	PN_q	&=\#\{\alpha \in \F_{q^n} \text{ is primitive normal}\}\\[.3cm]
	&=\delta_q(\alpha)\frac{\varphi (q^n-1)}{q^{n}}\cdot \frac{\Phi(x^n-1)}{q^{n}}\left (1+o(1)\right )q^n,\nonumber
\end{align}	
where $\delta_q(\alpha)>0$ is the density constant, which depends on the fixed element $\alpha\in \F_{q^n}$, {\color{red}\cite[Theorem 2]{CL1952A}}, \cite{SP2003}, \cite{PT2022}, et alia. This implies the asymptotic probability function

\begin{align} \label{eq4343PNEFF.050i12}
	P(\alpha \in \F_{q^n} \text{ is primitive normal})	&\gg \frac{1}{(\log \log q^n)^2}.
\end{align}	
This note provides a new result in the distribution of these elements.
\begin{thm} \label{thm4343PNEFF.050}\hypertarget{thm4343PNEFF.050}  The set of primitive normal elements in a finite field $\F_{q^n}$ is equidistributed in the finite field.
\end{thm}

 The background supporting materials appear in \hyperlink{S2727CEES-I}{Section} \ref{S2727CEES-I} to \hyperlink{S1900SP-S}{Section} \ref{S1900SP-S}. The proofs of the equidistribution of quadratic residues and primitive roots mod $p$, which are derived from the finite Fourier transforms (FFTs), are given in \hyperlink{S1900FFT-QN}{Section} \ref{S1900FFT-QN} and \hyperlink{S1900FFT-PE}{Section} \ref{S1900FFT-PE} respectively. Ultimately, the preceding results are spliced together to complete the proof of \hyperlink{thm4343PNEFF.050}{Theorem} \ref{thm4343PNEFF.050} in  \hyperlink{S1900FFT-PN}{Section} \ref{S1900FFT-PN}.

%SSSSSSSSSSSSSSSSSSSSSSSSSSSSSSSSSSSSSSSSSSSSSSSSSSSSSSSSSSSSSSSSSSSSSSSSSSSSSSSSSSS
%SSSSSSSSSSSSSSSSSSSSSSSSSSSSSSSSSSSSSSSSSSSSSSSSSSSSSSSSSSSSSSSSSSSSSSSSSSSSSSSSSSS
%SSSSSSSSSSSSSSSSSSSSSSSSSSSSSSSSSSSSSSSSSSSSSSSSSSSSSSSSSSSSSSSSSSSSSSSSSSSSSSSSSSS
%SSSSSSSSSSSSSSSSSSSSSSSSSSSSSSSSSSSSSSSSSSSSSSSSSSSSSSSSSSSSSSSSSSSSSSSSSSSSSSSSSSS
%SSSSSSSSSSSSSSSSSSSSSSSSSSSSSSSSSSSSSSSSSSSSSSSSSSSSSSSSSSSSSSSSSSSSSSSSSSSSSSSSSSS
%SSSSSSSSSSSSSSSSSSSSSSSSSSSSSSSSSSSSSSSSSSSSSSSSSSSSSSSSSSSSSSSSSSSSSSSSSSSSSSSSSSS
%SSSSSSSSSSSSSSSSSSSSSSSSSSSSSSSSSSSSSSSSSSSSSSSSSSSSSSSSSSSSSSSSSSSSSSSSSSSSSSSSSSS
\section{Estimates for Character Sums} \label{S2727CEES-I}\hypertarget{S2727CEES-I}
%TTTTTTTTTTTTTTTTTTTTTTTTTTTTTTTTTTTTTTTTTTTTTTTTTTTTTTTT
Some basic estimates of the incomplete character sums over the finite field  $\mathbb{F}_{q^n}$ are computed in this section. The technique employed here is essentially the same as those appearing in \cite{ES1957}, and \cite{GS2008}.

%LLLLLLLLLLLLLLLLLLLLLLLLLLLLLLLLLLLLLLLLLLLLLLLLLLLLLLLLLLLLLLLLLLLLLLLLLLLLLLLLLLLLLLL
%LLLLLLLLLLLLLLLLLLLLLLLLLLLLLLLLLLLLLLLLLLLLLLLLLLLLLLLLLLLLLLLLLLLLLLLLLLLLLLLLLLLLLLL
\begin{lem}   \label{lem2727CEES.675I}\hypertarget{lem2727CEES.675I} Let $q=p^k$ be a prime power and let $\mathcal{U},\mathcal{V}\subset \F_{q^n}$ be a pair of subsets of elements of cardinalities $\#\mathcal{U},\#\mathcal{V}\leq q^n$. If $\psi\ne1 $ is an additive character over the finite field $\F_{q^n}$, then
	\begin{equation} \label{eq2727CEES.650Id}
		\sum _{v\in \mathcal{V}}\sum _{u\in \mathcal{U}} \psi(uv) 
		\ll q^{n/2} \cdot \sqrt{\#U\cdot \#V }\nonumber.
	\end{equation} 
\end{lem} 

\begin{proof}[\textbf{Proof}]Let $\chi_{\scriptscriptstyle  \mathcal{U}} $ and $\chi_{\scriptscriptstyle\mathcal{V}}$ be the indicator functions of the subsets $\mathcal{U} \text{ and } \chi_{\mathcal{V}}\subset \F_{q^n}$ respectively. Now rewrite the character sum as
\begin{equation} \label{eq2727CEES.675Id}
\sum _{v\in \mathcal{V}}\sum _{u\in \mathcal{U}} \psi(uv) 
=\sum _{v\in \F_{q^n}}\chi_{\scriptscriptstyle\mathcal{V}}(v) \sum _{u\in \F_{q^n}}\chi_{\scriptscriptstyle\mathcal{U}}(u)  \psi(uv) ,
	\end{equation} 
and an application of the Schwarz inequality yields
\begin{equation} \label{eq2727CEES.675If}
	\left|	\sum _{v\in \mathcal{V}}\sum _{u\in \mathcal{U}} \psi(uv)\right|^2 
	\leq \sum _{v\in \F_{q^n}}\left|\chi_{\scriptscriptstyle\mathcal{V}}(v)\right)|^2  \sum _{v\in \F_{q^n}}\left|\sum _{u\in \F_{q^n}}\chi_{\scriptscriptstyle\mathcal{U}}(u)  \psi(uv)\right|^2  .
\end{equation} 

The norm of the outer sum is
\begin{equation} \label{eq2727CEES.675Ig}
	\sum _{v\in \F_{q^n}}\left|\chi_{\scriptscriptstyle\mathcal{V}}(v)\right)|^2 
	=\#V .
\end{equation}
And the norm of the inner sum is
\begin{eqnarray}\label{eq2727CEES.675Ii}
 \sum _{v\in \F_{q^n}}\left|\sum _{u\in \F_{q^n}}\chi_{\scriptscriptstyle\mathcal{U}}(u)  \psi(uv)\right|^2 
&=& \sum _{v\in \F_{q^n}}\sum _{u_1\in \F_{q^n}}\chi_{\scriptscriptstyle\mathcal{U}}(u_1)  \psi(u_1v)   \sum _{u_0\in \F_{q^n}}\overline{\chi_{\scriptscriptstyle\mathcal{U}}(u_0)  \psi(uv)}\\[.3cm]
&=&\sum _{u_1\in \F_{q^n}}\chi_{\scriptscriptstyle\mathcal{U}}(u_1)\sum _{u_0\in \F_{q^n}}\chi_{\scriptscriptstyle\mathcal{U}}(u_0)\sum _{v\in \F_{q^n}} \psi((u_1-u_0)v)\nonumber\\[.3cm]
&=&\sum _{u_0\in \F_{q^n}}\chi_{\scriptscriptstyle\mathcal{U}}(u_0)^2\sum _{v\in \F_{q^n}} 1\nonumber\\[.3cm]
	&=&q^n\#	\mathcal{U},\nonumber
\end{eqnarray}
since for $\psi\ne1$, the character sum
\begin{equation}\label{eq2727CEES.675Ij}
	\sum _{v\in \F_{q^n}} \psi((u_1-u_0)v)=\left \{
	\begin{array}{ll}
		q^n & \text{   \normalfont if } u_0=u_1,  \\
		0 & \text{   \normalfont if } u_0\ne u_1, \\
	\end{array} \right .
\end{equation}
and the indicator function $\chi_{\mathcal{U}}(u)=\overline{\chi_{\mathcal{U}}(u)}$ is a real value function. Merging \eqref{eq2727CEES.675Ig} and \eqref{eq2727CEES.675Ii} into \eqref{eq2727CEES.675If} returns
\begin{equation} \label{eq2727CEES.675Iv}
	\left|	\sum _{v\in \mathcal{V}}\sum _{u\in \mathcal{U}} \psi(uv)\right|^2 
	\leq q^n\#	\mathcal{U}\#V  .
\end{equation} 
This completes the verification.
\end{proof}

%LLLLLLLLLLLLLLLLLLLLLLLLLLLLLLLLLLLLLLLLLLLLLLLLLLLLLLLLLLLLLLLLLLLLLLLLLLLLL
%LLLLLLLLLLLLLLLLLLLLLLLLLLLLLLLLLLLLLLLLLLLLLLLLLLLLLLLLLLLLLLLLLLLLLLLLLLLLL
\begin{lem}   \label{lem2727CEES.650U}\hypertarget{lem2727CEES.650U} Let $\mathcal{U}$ be a subset of the finite field $\F_{q^n}$. If $\psi\ne1 $ be an additive character of the finite field, then
	\begin{equation} \label{eq2727CEES.650Ud}
		\sum _{u\in \mathcal{U}} \psi(uv) 
		\leq q^{n/2} .
	\end{equation} 
\end{lem} 

\begin{proof}[\textbf{Proof}] Let $\mathcal{V}\subseteq \F_{q^n}^{\times}$ be a subset. For each fixed $v\in \mathcal{V}$ the map $u\longrightarrow uv$ permutes the set $\mathcal{U}$. This implies that
	\begin{equation} \label{eq2727CEES.650Ui}
		\sum _{v\in \mathcal{V}}\sum _{u\in \mathcal{U}}  \psi(uv)=  \#\mathcal{V}\sum _{u\in \mathcal{U}} \psi(u).
	\end{equation}
	
	Consequently, applying \hyperlink{lem2727CEES.675I}{Lemma} \ref{lem2727CEES.675I} yields the upper bound
	\begin{eqnarray} \label{eq2727CEES.650Uj}
		q^{n/2} \cdot \sqrt{\#U\cdot \#V }&\geq& \left|	\sum _{v\in \mathcal{V}}\sum _{u\in \mathcal{U}} \psi(uv) \right|\\[.3cm]
		&=& \#\mathcal{V}	\left|	\sum _{u\in \mathcal{U}} \psi(u) \right|\nonumber.
	\end{eqnarray} 
	Setting $\mathcal{U}\subseteq \mathcal{V}$ completes the proof.
\end{proof}

%SSSSSSSSSSSSSSSSSSSSSSSSSSSSSSSSSSSSSSSSSSSSSSSSSSSSSSSSSSSSSSSSSSSSSSSSSSSSSSSSSSS
%SSSSSSSSSSSSSSSSSSSSSSSSSSSSSSSSSSSSSSSSSSSSSSSSSSSSSSSSSSSSSSSSSSSSSSSSSSSSSSSSSSS
%SSSSSSSSSSSSSSSSSSSSSSSSSSSSSSSSSSSSSSSSSSSSSSSSSSSSSSSSSSSSSSSSSSSSSSSSSSSSSSSSSSS
%SSSSSSSSSSSSSSSSSSSSSSSSSSSSSSSSSSSSSSSSSSSSSSSSSSSSSSSSSSSSSSSSSSSSSSSSSSSSSSSSSSS
%SSSSSSSSSSSSSSSSSSSSSSSSSSSSSSSSSSSSSSSSSSSSSSSSSSSSSSSSSSSSSSSSSSSSSSSSSSSSSSSSSSS
%SSSSSSSSSSSSSSSSSSSSSSSSSSSSSSSSSSSSSSSSSSSSSSSSSSSSSSSSSSSSSSSSSSSSSSSSSSSSSSSSSSS
%SSSSSSSSSSSSSSSSSSSSSSSSSSSSSSSSSSSSSSSSSSSSSSSSSSSSSSSSSSSSSSSSSSSSSSSSSSSSSSSSSSS
\section{Totient Functions over the Integers} \label{C5005}\hypertarget{C5005}
Let $p_1, p_2, p_3, \ldots, p_k$ be a sequence of prime numbers in increasing order, and let $n=p_1^{e_1} p_2^{e_2} \cdots p_t^{e_t}$ be the prime decomposition of an arbitrary integer.  
\begin{dfn}\label{dfn5005.100d}\hypertarget{dfn5005.100d}
	The Euler totient function over the finite ring $\mathbb{Z}/n\mathbb{Z}$ is defined by  $$\varphi (n)=\sum_{\substack{k<n\\\gcd(k,n)=1}}1=\prod_{p\mid n}\left( 1-\frac{1}{p}\right).$$ It counts the number of relatively prime integers up to $n\geq1$.
\end{dfn}

A result for the ubiquitous lower bound of the totient ratio $\varphi(n)/n$ appears in {\color{red}\cite[Theorem 2.9]{MV2007}}. And an explicit lower bound for all integers $n\geq10$ is given below. 

\begin{lem}  \label{lem9955P.400TL}\hypertarget{lem9955P.400TL} If $n\geq10$ is a large integer, then
	$$\frac{\varphi(n)}{n}\geq\frac{3}{e^{\gamma}\pi^2}\frac{1}{\log \log n},$$
	where $\gamma>0$ is Euler constant.
\end{lem}

%SSSSSSSSSSSSSSSSSSSSSSSSSSSSSSSSSSSSSSSSSSSSSSSSSSSSSSSSSSSSSSSSSSSSSSSSSSSSSSSSSSS
%SSSSSSSSSSSSSSSSSSSSSSSSSSSSSSSSSSSSSSSSSSSSSSSSSSSSSSSSSSSSSSSSSSSSSSSSSSSSSSSSSSS
%SSSSSSSSSSSSSSSSSSSSSSSSSSSSSSSSSSSSSSSSSSSSSSSSSSSSSSSSSSSSSSSSSSSSSSSSSSSSSSSSSSS
%SSSSSSSSSSSSSSSSSSSSSSSSSSSSSSSSSSSSSSSSSSSSSSSSSSSSSSSSSSSSSSSSSSSSSSSSSSSSSSSSSSS
%SSSSSSSSSSSSSSSSSSSSSSSSSSSSSSSSSSSSSSSSSSSSSSSSSSSSSSSSSSSSSSSSSSSSSSSSSSSSSSSSSSS
%SSSSSSSSSSSSSSSSSSSSSSSSSSSSSSSSSSSSSSSSSSSSSSSSSSSSSSSSSSSSSSSSSSSSSSSSSSSSSSSSSSS
%SSSSSSSSSSSSSSSSSSSSSSSSSSSSSSSSSSSSSSSSSSSSSSSSSSSSSSSSSSSSSSSSSSSSSSSSSSSSSSSSSSS
\section{Totient Functions over Dedekind Domains}\label{S7171TFDD-A}\hypertarget{S7171TFDD-A}
Let $q=p^m$ be a prime power and $\mathcal{R}_q=\F_q[x]/\langle f(x)\rangle$. If $g(x)\in \F_q[x]$, then the totient function over the finite ring $\mathcal{R}_q$ is defined as follows. 
\begin{align}\label{eq1771TFDD.200exei2}
	\Phi_q(g(x))&=\#\{r(x)\in \F_q[x]:\gcd(g,r)=c\}\\[.2cm]
	&=N(g(x))\prod_{r(x) \mid  g(x)}\left (1-\frac{1}{N(r)}\right ),\nonumber 
\end{align}		
where the polynomial $r(x)$ ranges over the irreducible factors of $g(x)$, $N(g)=q^{\deg g}$ is the norm of $g(x)$ and $c\in \F_q^{\times}$ is a constant.

%EEEEEEEEEEEEEEEEEEEEEEEEEEEEEEEEEEEEEEEEEEEEEEEEEEEEEEEEEEEEEEEEEEEEEEE
%EEEEEEEEEEEEEEEEEEEEEEEEEEEEEEEEEEEEEEEEEEEEEEEEEEEEEEEEEEEEEEEEEEEEEEE

%LLLLLLLLLLLLLLLLLLLLLLLLLLLLLLLLLLLLLLLLLLLLLLLLLLLLLLLLLLLLLLLLLLLLLLLLLLLLLLLL
%LLLLLLLLLLLLLLLLLLLLLLLLLLLLLLLLLLLLLLLLLLLLLLLLLLLLLLLLLLLLLLLLLLLLLLLLLLLLLLLL
\begin{lem} \label{lem1771TFDD.300FF}\hypertarget{lem1771TFDD.300FF} If $q$ is an odd prime power and $n\geq2$, then the finite field $\F_{q^n}$ contains 
	$$\Phi(x^n-1)=q^n\prod_{d(x)\mid x^n-1}\left(1-\frac{1}{q^{\deg d(x)}} \right) $$	$\Phi(x^n-1)$ normal elements.
\end{lem}

Earlier applications of the example above appears in {\color{red}\cite[Theorem 11]{OO1934}}, \cite{CL1952A}, \cite{DH1968}, {\color{red}\cite[p.\;219]{LS1987}}, et alia. A completely detailed proof of the normal base theorem appears in {\color{red}\cite[Theorem 1]{HT2018}}. Apparently, the normal base theorem for finite fields was conjecture by  Eisenstein and Schonemann circa 1850. The complete proof was given by Hensel in \cite{HK1888}. \\

%CCCCCCCCCCCCCCCCCCCCCCCCCCCCCCCCCCCCCCCCCCCCCCCCCCCCCCCCCCCCCCCCCCCCCCCCCCCCCCC
%CCCCCCCCCCCCCCCCCCCCCCCCCCCCCCCCCCCCCCCCCCCCCCCCCCCCCCCCCCCCCCCCCCCCCCCCCCCCCCC
Assuming $n\mid q-1$, the probability of normal elements in finite field has the exact form
\begin{align}\label{eq1771TFDD.300FF0}
	\frac{\Phi(x^n-1)}{q^n-1}
	&=\frac{q^n}{q^n-1}\left(1-\frac{1}{q} \right)^{n
	}\\[.3cm]
	&=\frac{q^n}{q^n-1}e^{n\log \left(1-\frac{1}{q} \right)}\nonumber\\[.3cm]
	&=1-\frac{n}{ q}+O\left (\frac{n(n-1}{ q^2}\right )\nonumber.
\end{align} 
A different and more practical asymptotic for the lower bound suitable for all $n\in [2,q-1]$ and in terms of $\log q^n$ is provided here.

\begin{cor} \label{cor1771TFDD.300FF-A}\hypertarget{cor1771TFDD.300FF-A} Let $q=p^k\geq 2^k$ be a prime power and let $n\geq2$. Then $$\frac{\Phi(x^n-1)}{q^n-1}\gg \frac{1}{\log 5\log q^{n}} $$
	uniformly for all $p\geq 2$, $k\geq 1$ and $n\geq 2.$
\end{cor}

\begin{proof} There are two important cases indexed by $n\mid q-1$ and $n\nmid q-1$. The expression in \eqref{eq1771TFDD.300FF0} illustrates that to derive a lower bound in terms of $\log q^n$, it is sufficient to consider the first case. Toward this goal, suppose that $n\mid q-1$. This parameter specifies the lowest bound possible. Furthermore, it implies that $n<q$, the finite field contains all the $n$th roots of unity, and the polynomial $x^n-1\in \F_q[x]$ splits into linear factors. By \hyperlink{lem1771TFDD.300FF}{Lemma} \ref{lem1771TFDD.300FF} there is a lower bound of the form
	\begin{align}\label{eq1771TFDD.300FF1}
		\frac{\Phi(x^n-1)}{q^n-1}&=\frac{q^n}{q^n-1}\prod_{d(x)\mid x^n-1}\left(1-\frac{1}{q^{\deg d(x)}} \right)\\[.3cm]
		&=\frac{q^n}{q^n-1}\left(1-\frac{1}{q} \right)^{n}\nonumber\\[.3cm]
		&\gg\frac{1}{\log 5\log q^n}\nonumber.
	\end{align}
	Set $x=q>0$ and let $n\in [2,x-1]$ be an integer. Then, the conclusion follows from the nonnegativity of the function
	\begin{equation}
		f(n,x)=\left(1-\frac{1}{x} \right)^{n}-\frac{1}{5\log 5\log x^{n}}>0
	\end{equation}
	over the domain $[2,x-1]\times [2,\infty]$. 
\end{proof}

%SSSSSSSSSSSSSSSSSSSSSSSSSSSSSSSSSSSSSSSSSSSSSSSSSSSSSSSSSSSSSSSSSSSSSSSSSSSS
%SSSSSSSSSSSSSSSSSSSSSSSSSSSSSSSSSSSSSSSSSSSSSSSSSSSSSSSSSSSSSSSSSSSSSSSSSSSS
%SSSSSSSSSSSSSSSSSSSSSSSSSSSSSSSSSSSSSSSSSSSSSSSSSSSSSSSSSSSSSSSSSSSSSSSSSSSS
%SSSSSSSSSSSSSSSSSSSSSSSSSSSSSSSSSSSSSSSSSSSSSSSSSSSSSSSSSSSSSSSSSSSSSSSSSSSS
%SSSSSSSSSSSSSSSSSSSSSSSSSSSSSSSSSSSSSSSSSSSSSSSSSSSSSSSSSSSSSSSSSSSSSSSSSSSS
%SSSSSSSSSSSSSSSSSSSSSSSSSSSSSSSSSSSSSSSSSSSSSSSSSSSSSSSSSSSSSSSSSSSSSSSSSSSS
%SSSSSSSSSSSSSSSSSSSSSSSSSSSSSSSSSSSSSSSSSSSSSSSSSSSSSSSSSSSSSSSSSSSSSSSSSSSS
\section{Characteristic Functions for Primitive Elements}
\label{S7171CFPE-P} \hypertarget{S7171CFPE-P}
A primitive element generates the $\Z$-module

\begin{equation}\label{eq7171CFNE.100a}
	\F_{q^n}^{\times}\cong	\Z/(q^n-1)\Z.
\end{equation}	
The group of units $\left( \Z/(q^n-1)\Z\right)^{\times}$ has precisely $\varphi(q^n-1)$ units, and each unit is a primitive element. There several possible techniques that can be used to construct characteristic functions of primitive elements in the group of units $\left( \Z/(q^n-1)\Z\right)^{\times}$. The standard characteristic function for primitive elements in finite field $\F_{q^n}$ dependents on the factorization of the integer $q^n-1$. Whereas a new divisor-free characteristic function for primitive elements introduced here is not dependent on the factorization of the integer $q^n-1$. The basic analytic principles of these indicator functions are presented below. 

%SSSSSSSSSSSSSSSSSSSSSSSSSSSSSSSSSSSSSSSSSSSSSSSSSSSSSSSSSSSSSSSS
%SSSSSSSSSSSSSSSSSSSSSSSSSSSSSSSSSSSSSSSSSSSSSSSSSSSSSSSSSSSSSSSS
%SSSSSSSSSSSSSSSSSSSSSSSSSSSSSSSSSSSSSSSSSSSSSSSSSSSSSSSSSSSSSSSS
%SSSSSSSSSSSSSSSSSSSSSSSSSSSSSSSSSSSSSSSSSSSSSSSSSSSSSSSSSSSSSSSS
%SSSSSSSSSSSSSSSSSSSSSSSSSSSSSSSSSSSSSSSSSSSSSSSSSSSSSSSSSSSSSSSS
%SSSSSSSSSSSSSSSSSSSSSSSSSSSSSSSSSSSSSSSSSSSSSSSSSSSSSSSSSSSSSSSS
\subsection{Divisor Dependent Characteristic Functions for Primitive Elements}
%\label{S7171CFPE-B} \hypertarget{S7171CFPE-B}
%\subsection{Divisors Dependent Characteristic Function}
A representation of the characteristic function dependent on the orders of the cyclic groups is given below. This representation is sensitive to the primes decompositions of the order of the cyclic group $q^n-1=\#\F_{q^n}^{\times}$. 

\begin{lem} \label{lem2727CFFR.100-F}\hypertarget{lem2727CFFR.100-F} If $q=p^k$ is a prime power and $\alpha\in \F_{q^n}$ is an invertible element in the cyclic group $\F_{q^n}^{\times}$, then
	\begin{equation}\label{eq2727CFFR.100d}
		\Psi (\alpha)=\frac{\varphi (q^n-1)}{q^n-1}\sum _{d \mid q^n-1} \frac{\mu (d)}{\varphi (d)}\sum _{\ord(\chi ) = d} \chi (\alpha)=
		\left \{\begin{array}{ll}
			1 & \text{ if } \ord_q (\alpha)=q^n-1,  \\
			0 & \text{ if } \ord_q (\alpha)\neq q^n-1. \\
		\end{array} \right .
	\end{equation}
\end{lem}

There are a few other variant proofs of this result, these are widely available in the literature, the proof given in {\color{red}\cite[p.\;221 ]{LS1987}} has an error. Almost every result in the theory of primitive roots in finite fields is based on this characteristic function, but sometimes written in different forms. \\

The authors in \cite{DH1937}, \cite{WR2001} attribute the simplest case for prime finite fields $\F_{p}$ of this formula to Vinogradov, \cite{VI1927}, and other authors attribute this formula to Landau, \cite{LE1927}. The proof and other details on the characteristic function are given in {\color{red}\cite[p. 863]{ES1957}}, {\color{red}\cite[p.\ 258]{LN1997}}, et alii.

%SSSSSSSSSSSSSSSSSSSSSSSSSSSSSSSSSSSSSSSSSSSSSSSSSSSSSSSSSSSSSSSS
%SSSSSSSSSSSSSSSSSSSSSSSSSSSSSSSSSSSSSSSSSSSSSSSSSSSSSSSSSSSSSSSS
%SSSSSSSSSSSSSSSSSSSSSSSSSSSSSSSSSSSSSSSSSSSSSSSSSSSSSSSSSSSSSSSS
%SSSSSSSSSSSSSSSSSSSSSSSSSSSSSSSSSSSSSSSSSSSSSSSSSSSSSSSSSSSSSSSS
\subsection{Divisorfree Characteristic Functions for Primitive Elements}
\label{S7171CFPE-B} \hypertarget{S7171CFPE-B}
Let \(\tau\) be a primitive root in $\F_{q^n}$, let $\log_{\tau}\alpha$ be the discrete logarithm with respect to $\tau$, and $s\in \mathcal{S}=\{s<q^n:\gcd(s,q^n-1)=1\}$. The discrete logarithm is defined by the map
\begin{align}\label{eq7171CFNE.150PDLM}
	\F_{q^n}^{\times}&\longrightarrow \left(  \Z(q^n-1)/\Z\right)^{\times}\\[.3cm]
	\alpha &\longrightarrow \log_{\tau}\alpha.
\end{align}
A new \textit{divisors-free} representation of the characteristic function of primitive element is introduced in this section. It detects the order $\ord_{q^n} \alpha$
of the element $\alpha \in \F_{q^n}$ by means of the solutions of the equation \begin{equation}\label{eq7171CFNE.150PDF0}
	s-\log_{\tau}\alpha=0.
\end{equation}  

\begin{lem} \label{lem7171CFPE.150PDF} \hypertarget{lem7171CFNE.150PDF}  Let $q=p^k$ be a prime power, let \(\tau\) be a primitive root in $\F_{q^n}$ and let \(\psi \neq 1\) be a nonprincipal additive character of order $\ord  \psi =q^n$. If $\alpha \in \F_{q^n}$ is a nonzero element, then
	\begin{eqnarray}\label{eq7171CFPE.150PDFd}
		\Psi (\alpha)&=&\sum _{\substack{1\leq s\leq q^n-1\\\gcd (s,q^n-1)=1}} \frac{1}{q^n}\sum _{0\leq t\leq q^n-1} \psi \left ((s-\log_{\tau}\alpha)t\right)\\[.2cm]
		&=&\left \{
		\begin{array}{ll}
			1 & \text{   \normalfont if } \ord_{q^n} (\alpha)=q^n-1,  \\[.2cm]
			0 & \text{   \normalfont if } \ord_{q^n} (\alpha)\ne q^n-1. \\
		\end{array} \right .\nonumber
	\end{eqnarray}
\end{lem}	

\begin{proof}[\textbf{Proof}] A detailed proof appears in {\color{red}\cite[Lemma 8.2]{CN2025}}.
\end{proof}

Other versions of the divisorfree characteristic function for primitive elements are possible.
%SSSSSSSSSSSSSSSSSSSSSSSSSSSSSSSSSSSSSSSSSSSSSSSSSSSSSSSSSSSSSSSSSSSSSSSSSSSS
%SSSSSSSSSSSSSSSSSSSSSSSSSSSSSSSSSSSSSSSSSSSSSSSSSSSSSSSSSSSSSSSSSSSSSSSSSSSS
%SSSSSSSSSSSSSSSSSSSSSSSSSSSSSSSSSSSSSSSSSSSSSSSSSSSSSSSSSSSSSSSSSSSSSSSSSSSS
%SSSSSSSSSSSSSSSSSSSSSSSSSSSSSSSSSSSSSSSSSSSSSSSSSSSSSSSSSSSSSSSSSSSSSSSSSSSS
%SSSSSSSSSSSSSSSSSSSSSSSSSSSSSSSSSSSSSSSSSSSSSSSSSSSSSSSSSSSSSSSSSSSSSSSSSSSS
%SSSSSSSSSSSSSSSSSSSSSSSSSSSSSSSSSSSSSSSSSSSSSSSSSSSSSSSSSSSSSSSSSSSSSSSSSSSS
%SSSSSSSSSSSSSSSSSSSSSSSSSSSSSSSSSSSSSSSSSSSSSSSSSSSSSSSSSSSSSSSSSSSSSSSSSSSS
\section{Characteristic Functions for Normal Elements} \label{S7171CFNE-N}\hypertarget{S7171CFNE-N}
A normal element generates the $\F_q$-module

\begin{equation}\label{eq7171CFNE.150a}
	\F_{q^n}\cong\F_{q}[x]/(x^n-1)\F_{q}.
\end{equation}	
The group of units $\left( \F_{q}[x]/(x^n-1)\F_{q}\right)^{\times}$ has precisely $\Phi(x^n-1)$ units, and each unit is a normal element. There several possible techniques that can be used to construct characteristic functions of normal elements in the group of units $\left( \F_{q}[x]/(x^n-1)\F_{q}\right)^{\times}$. The standard characteristic function for normal elements dependents on the factorization of the polynomial $x^n-1$. Whereas a new divisor-free characteristic function for normal elements introduced here is not dependent on the factorization of the polynomial $x^n-1$. The basic analytic principles of these indicator functions are presented below. 
\subsection{Divisors Dependent Characteristic Functions for Normal Elements} \label{S7171CFNE-C}\hypertarget{S7171CFNE-C}
%\subsection{Divisors Dependent Characteristic Function}
A representation of the characteristic function of normal elements $\eta\in \F_q[x]/f(x)$ in finite rings is outline here. This representation is sensitive to the irreducible decompositions $x^n-1=p_1(x)^{e_1}p_2(x)^{e_2}\cdots p_k(x)^{e_k}$, with $p_i(x)\in\F_q[x]$ irreducible and $e_i\geq1$. 

\begin{lem} \label{lem7171CFNE.150b}\hypertarget{lem7171CFNE.150b}
	If \(\alpha\in \F_{q^n}\) be a nonzero element, then
	\begin{align}\label{eq7171CFNE.150d}
		\Psi_q (\alpha)&=\frac{\Phi_q (x^n-1)}{q^n}\sum _{d(x) \mid x^n-1} \frac{\mu_q (d(x))}{\Phi_q (d(x))}\sum _{\Ord(\psi ) = d(x)} \psi (\alpha)\\[.2cm]
		&=
		\left \{\begin{array}{ll}
			1 & \text{  \normalfont if } \Ord_q (\alpha)=x^n-1,  \\[.2cm]
			0 & \text{  \normalfont if } \Ord_q (\alpha)\neq x^n-1. \\
		\end{array} \right .\nonumber 
	\end{align}
\end{lem}

The earliest development of this indicator function seems to be {\color{red}\cite[Theorem 11]{OO1934}} and {\color{red}\cite[Lemma 4]{CL1952A}}.
%SSSSSSSSSSSSSSSSSSSSSSSSSSSSSSSSSSSSSSSSSSSSSSSSSSSSSSSSSSSSSSSS
%SSSSSSSSSSSSSSSSSSSSSSSSSSSSSSSSSSSSSSSSSSSSSSSSSSSSSSSSSSSSSSSS
%SSSSSSSSSSSSSSSSSSSSSSSSSSSSSSSSSSSSSSSSSSSSSSSSSSSSSSSSSSSSSSSS
%SSSSSSSSSSSSSSSSSSSSSSSSSSSSSSSSSSSSSSSSSSSSSSSSSSSSSSSSSSSSSSSS
%SSSSSSSSSSSSSSSSSSSSSSSSSSSSSSSSSSSSSSSSSSSSSSSSSSSSSSSSSSSSSSSS
%SSSSSSSSSSSSSSSSSSSSSSSSSSSSSSSSSSSSSSSSSSSSSSSSSSSSSSSSSSSSSSSS
%SSSSSSSSSSSSSSSSSSSSSSSSSSSSSSSSSSSSSSSSSSSSSSSSSSSSSSSSSSSSSSSS
\subsection{Divisorfree Characteristic Functions for Normal Elements} %\label{S7171CFNE-N}\hypertarget{S7171CFNE-N}
%\subsection{Divisors Dependent Characteristic Function}
A new divisor-free representation of the characteristic function of normal elements $\eta\in \F_q[x]/f(x)$, where $f(x)\in\F_{q}[x]$ is a polynomial of degree $\deg f=n$, in finite rings is outline here. This representation is not sensitive to the irreducible decompositions $x^n-1=p_1(x)^{e_1}p_2(x)^{e_2}\cdots p_k(x)^{e_k}$, with $p_i(x)\in\F_q[x]$ irreducible and $e_i\geq1$. The result is expressed in terms of the discrete logarithm, defined in \eqref{eq7171CFNE.150PDLM} in \hyperlink{S7171CFPE-P}{Section } \ref{S7171CFPE-P} and a nontrivial additive character $\psi(t)=e^{i2\pi at/q^n}$ with $a\ne0$.

\begin{lem} \label{lem7171CFNE.150NDF}\hypertarget{lem7171CFNE.150NDF} Let $q=p^k$ be a prime power and let $\eta \in \F_{q^n}$ be a normal element. If \(\alpha\in \F_{q^n}\) is a nonzero element, then
	\begin{align}\label{eq7171CFNE.150C}
		\Psi_q (\alpha)&=\sum _{\substack{\deg s(x)\leq n-1\\\gcd (s(x),x^n-1)=1}} \frac{1}{q^n}\sum _{0\leq t\leq q^n-1} e^{ \frac{i2\pi(\log_{\tau}s(x)\circ\eta-\log_{\tau}\alpha)t}{q^n}}\\[.2cm]
		&=
		\left \{\begin{array}{ll}
			1& \text{  \normalfont if } \Ord_q (\alpha)=x^n-1,  \\[.2cm]
			0 & \text{  \normalfont if } \Ord_q (\alpha)\neq x^n-1. \\
		\end{array} \right .\nonumber 
	\end{align}
	
\end{lem}

\begin{proof}[\textbf{Proof}] A detailed proof appears in {\color{red}\cite[Lemma 9.2]{CN2025}}.
\end{proof}
	
Other versions of the divisorfree characteristic function for normal elements are possible.

%SSSSSSSSSSSSSSSSSSSSSSSSSSSSSSSSSSSSSSSSSSSSSSSSSSSSSSSSSSSSSSSSSSSSSSSSSSSSSSSSSSS
%SSSSSSSSSSSSSSSSSSSSSSSSSSSSSSSSSSSSSSSSSSSSSSSSSSSSSSSSSSSSSSSSSSSSSSSSSSSSSSSSSSS
%SSSSSSSSSSSSSSSSSSSSSSSSSSSSSSSSSSSSSSSSSSSSSSSSSSSSSSSSSSSSSSSSSSSSSSSSSSSSSSSSSSS
%SSSSSSSSSSSSSSSSSSSSSSSSSSSSSSSSSSSSSSSSSSSSSSSSSSSSSSSSSSSSSSSSSSSSSSSSSSSSSSSSSSS
%SSSSSSSSSSSSSSSSSSSSSSSSSSSSSSSSSSSSSSSSSSSSSSSSSSSSSSSSSSSSSSSSSSSSSSSSSSSSSSSSSSS
%SSSSSSSSSSSSSSSSSSSSSSSSSSSSSSSSSSSSSSSSSSSSSSSSSSSSSSSSSSSSSSSSSSSSSSSSSSSSSSSSSSS
%SSSSSSSSSSSSSSSSSSSSSSSSSSSSSSSSSSSSSSSSSSSSSSSSSSSSSSSSSSSSSSSSSSSSSSSSSSSSSSSSSSS

%SSSSSSSSSSSSSSSSSSSSSSSSSSSSSSSSSSSSSSSSSSSSSSSSSSSSSSSSSSSSSSSS
%SSSSSSSSSSSSSSSSSSSSSSSSSSSSSSSSSSSSSSSSSSSSSSSSSSSSSSSSSSSSSSSS
%SSSSSSSSSSSSSSSSSSSSSSSSSSSSSSSSSSSSSSSSSSSSSSSSSSSSSSSSSSSSSSSS
%SSSSSSSSSSSSSSSSSSSSSSSSSSSSSSSSSSSSSSSSSSSSSSSSSSSSSSSSSSSSSSSS
%SSSSSSSSSSSSSSSSSSSSSSSSSSSSSSSSSSSSSSSSSSSSSSSSSSSSSSSSSSSSSSSS
%SSSSSSSSSSSSSSSSSSSSSSSSSSSSSSSSSSSSSSSSSSSSSSSSSSSSSSSSSSSSSSSS
%SSSSSSSSSSSSSSSSSSSSSSSSSSSSSSSSSSSSSSSSSSSSSSSSSSSSSSSSSSSSSSSS
\section{Finite Fourier Transforms}
\label{S5555FFT}\hypertarget{5555FFT}
%\subsection{Finite Fourier Transform}
Let $f: \C \longrightarrow \C$ be a function, and let $q \in \N$ be a large integer. 

\begin{dfn}\label{dfn5555FFT.300}\hypertarget{dfn5555FFT.300} {\normalfont 
		The discrete Fourier transform of the function $f:\N\longrightarrow \C$ and its inverse are defined by
		\begin{equation} \label{eq5555FFT.300d}
			\hat{f}(s)=\sum_{0 \leq t\leq q-1} f(t)e^{i \pi st/q}
		\end{equation}
		and 
		\begin{equation}\label{5555FFT.300c}
			f(t)=\frac{1}{q}\sum_{0 \leq s\leq q-1}\hat{f}(s)e^{-i2\pi st/q},
		\end{equation}
		respectively.		
	}
\end{dfn} 

The finite Fourier transform and its inverse are used here to derive a summation kernel function, which is almost identical to the Dirichlet kernel, in this application $q=p$ is a prime number.

\begin{dfn} \label{dfn5555FFT.100} \hypertarget{dfn5555FFT.100}{\normalfont Let $ p$ be a prime, let $\omega=e^{i 2 \pi/p}$ be a root of unity. The \textit{finite summation kernel} is defined by the finite Fourier transform identity
		\begin{equation} \label{eq5555FFT.100g}
			\mathcal{K}(f(n))=\frac{1}{p} \sum_{0 \leq t\leq p-1,}  \sum_{0 \leq s\leq p-1} \omega^{t(n-s)}f(s)=f(n).\end{equation}
	} 
\end{dfn}

\begin{thm}\label{thm5555FFT.100PT} \hypertarget{thm5555FFT.100PT}  {\normalfont(Parseval theorem)} Let $f,\;g:\mathbb{F}_q\longrightarrow\mathbb{C}$ abe a pair of functions and let $\widehat{f},\;\widehat{g}:\mathbb{F}_q\longrightarrow\mathbb{C}$ be the corresponding finite Fourier transforms. Then
	\begin{equation}
		\sum_{s\in\mathbb{F}_q}	\widehat{f}(s)\overline{\widehat{g}(s)} = q\sum_{t\in\mathbb{F}_q}f(t)\overline{g(t)},
	\end{equation}
	where $\overline{f(t)}$ is the complex conjugate of $f(t)$.
\end{thm}

The special case for $f(t)=g(t)$ is frequently used in applications.

\begin{thm}\label{thm5555FFT.100SPT} \hypertarget{thm5555FFT.100SPT} {\normalfont(Plancherel theorem)} Let $f:\mathbb{F}_q\longrightarrow\mathbb{C}$ be a of function and let $\widehat{f}:\mathbb{F}_q\longrightarrow\mathbb{C}$ be its corresponding finite Fourier transform. Then
	\begin{equation}
		\sum_{s\in\mathbb{F}_q}	| \widehat{f}(s)|^2 = q\sum_{t\in\mathbb{F}_q}|f(t)|^2.
	\end{equation}
\end{thm}
%SSSSSSSSSSSSSSSSSSSSSSSSSSSSSSSSSSSSSSSSSSSSSSSSSSSSSSSSSSSSSSSSSSSSSSSSSSSSS
%SSSSSSSSSSSSSSSSSSSSSSSSSSSSSSSSSSSSSSSSSSSSSSSSSSSSSSSSSSSSSSSSSSSSSSSSSSSSS
%SSSSSSSSSSSSSSSSSSSSSSSSSSSSSSSSSSSSSSSSSSSSSSSSSSSSSSSSSSSSSSSSSSSSSSSSSSSSS
%SSSSSSSSSSSSSSSSSSSSSSSSSSSSSSSSSSSSSSSSSSSSSSSSSSSSSSSSSSSSSSSSSSSSSSSSSSSSS
%SSSSSSSSSSSSSSSSSSSSSSSSSSSSSSSSSSSSSSSSSSSSSSSSSSSSSSSSSSSSSSSSSSSSSSSSSSSSS
%SSSSSSSSSSSSSSSSSSSSSSSSSSSSSSSSSSSSSSSSSSSSSSSSSSSSSSSSSSSSSSSSSSSSSSSSSSSSS
%SSSSSSSSSSSSSSSSSSSSSSSSSSSSSSSSSSSSSSSSSSSSSSSSSSSSSSSSSSSSSSSSSSSSSSSSSSSSS
\section{Finite Fourier Transform and Salem Sets}\label{S1900SP-S}\hypertarget{S11900SP-S}The  finite Fourier transform of the characteristic of a sequence of real numbers is very closely link to the Weyl equidistribution criterion of the sequence of real numbers. The definition of Salem set specifies certain classification of equidistribution based on the norm of the FFT. The finite Fourier transform of the characteristic function $\widehat{E}(s)$ of a subset $E\subset \mathbb{F}_{q^n}$ of cardinality $\#E$ has the trivial upper bound $|\widehat{E}(s)|\leq \#E$.
\begin{dfn}\label{dfn11900SP.200}\hypertarget{dfn11900SP.200}{\normalfont 
		A set $E\subset \mathbb{F}_{q^n}$ is called a \textit{Salem set} if the cardinality $\#E$ and the finite Fourier transform of the characteristic function $\widehat{E}(s)$ satisfy the norm relation
		\begin{equation}
			\sup_{s\ne0}| \widehat{E}(s)| \ll (\#E)^{1/2}.
		\end{equation}
		and such sets $E$ should be thought of as being optimal from a Fourier-analytic point of view. This sets are as random or unstructured as possible.}
\end{dfn}	
The corresponding finite sum attached to the FFT exhibits maximal or nearly maximal cancellation. Thus, the absolute values of the coefficients $\widehat{E(s)}$, $s\in \mathbb{F}_{q^n}^{\times}$, of the FFT are small except at zero or average value $\widehat{E(0)}=\#E$ of the FFT. A variety of Salem sets are demonstrated in \cite{IR2007}, {\color{red}\cite[Page 4.]{FR2025}}

%SSSSSSSSSSSSSSSSSSSSSSSSSSSSSSSSSSSSSSSSSSSSSSSSSSSSSSSSSSSSSSSSSSSSSSSSSSSSS
%SSSSSSSSSSSSSSSSSSSSSSSSSSSSSSSSSSSSSSSSSSSSSSSSSSSSSSSSSSSSSSSSSSSSSSSSSSSSS
%SSSSSSSSSSSSSSSSSSSSSSSSSSSSSSSSSSSSSSSSSSSSSSSSSSSSSSSSSSSSSSSSSSSSSSSSSSSSS
%SSSSSSSSSSSSSSSSSSSSSSSSSSSSSSSSSSSSSSSSSSSSSSSSSSSSSSSSSSSSSSSSSSSSSSSSSSSSS
%SSSSSSSSSSSSSSSSSSSSSSSSSSSSSSSSSSSSSSSSSSSSSSSSSSSSSSSSSSSSSSSSSSSSSSSSSSSSS
%SSSSSSSSSSSSSSSSSSSSSSSSSSSSSSSSSSSSSSSSSSSSSSSSSSSSSSSSSSSSSSSSSSSSSSSSSSSSS
%SSSSSSSSSSSSSSSSSSSSSSSSSSSSSSSSSSSSSSSSSSSSSSSSSSSSSSSSSSSSSSSSSSSSSSSSSSSSS
\section{Finite Fourier Transforms of the Set of Quadratic Residues} \label{S1900FFT-QN}\hypertarget{S1900FFT-QN}
The structural properties of the finite Fourier transform of the characteristic functions of various sets of quadratic residues and quadratic nonresidues are investigated in this section. 
%PPPPPPPPPPPPPPPPPPPPPPPPPPPPPPPPPPPPPPPPPPPPPPPPPPPPPPPPPPPPPPPPPPPPPPPPPPPPP
%PPPPPPPPPPPPPPPPPPPPPPPPPPPPPPPPPPPPPPPPPPPPPPPPPPPPPPPPPPPPPPPPPPPPPPPPPPPPP
\begin{prop}   \label{prop1900FCQ.300}\hypertarget{prop1900FCQ.300} If $p$ is a large prime, then the finite Fourier transform $\widehat{\kappa(s)}$ of the characteristic function $\Psi(s)$ of quadratic residues is basically a two values function
	\begin{equation}\label{eq1900FCQ.300m}
		\widehat{\Psi(s)}=
		\begin{cases}
			(p-1)/2 & \text{if } s=0,\\ 
			O(p^{1/2})& \text{if } s\ne0,\\ 
		\end{cases}
	\end{equation}
	where $s\in\F_p$.
\end{prop} 
\begin{proof}[\textbf{Proof}] By \hyperlink{dfn5555FFT.300}{Definition} \ref{dfn5555FFT.300}, the finite Fourier transform $\widehat{\kappa}:\F_p^{\times}\longrightarrow \C$ of the characteristic function 
\begin{equation}
		\kappa(t)=\frac{1}{2}\left( 1+\left( \frac{t}{p}\right) \right)
	\end{equation} 
	of quadratic residues mod $p$, is the followings:
	\begin{align}\label{eq1900FCQ.300e}
		\widehat{\kappa (s)}&=	\sum_{0\leq t\leq p-1}\kappa (t)e^{\frac{i2\pi st}{p}}\\[.3cm]
		&=\sum_{1\leq t\leq p-1}e^{\frac{i2\pi st}{p}}\cdot \frac{1}{2}\left( 1+\left( \frac{t}{p}\right) \right) \nonumber\\[.3cm]
		&=\frac{1}{2}\sum_{0\leq t\leq p-1}e^{\frac{i2\pi st}{p}}+\frac{1}{2}\sum_{0\leq t\leq p-1}e^{\frac{i2\pi st}{p}}\left( \frac{t}{p}\right)\nonumber, 
	\end{align}
	where $s\in [1,p-1]$, $\kappa (0)=0$, $\widehat{\kappa (0)}=(p-1)/2$ and $\sum_{0\leq t\leq p-1}e^{\frac{i2\pi st}{p}}=0$. 
	Set $\psi(st)=e^{\frac{i2\pi st}{p}}$ and let $\chi(t)=(t\,|\,p)$ be the quadratic symbol mod $p$, then the inner sum corresponds to a Gauss sum
	up to a complex constant:
	\begin{align}\label{eq1900FCQ.300h}
		\sum_{1\leq t\leq p-1}e^{\frac{i2\pi st}{p}}\chi(t) =\chi(s^{-1})\sum_{1\leq z\leq p-1}\psi(z)\chi(z)=G(\psi,\chi)\chi(s^{-1})
	\end{align}
	and $|G(\psi,\chi)\chi(s^{-1})|= p^{1/2}$. In view of this observation, it reduces to
	\begin{align}\label{eq1900FCQ.300k}
		\widehat{\kappa(s)}&=\frac{1}{2}\left( \frac{s^{-1}}{p}\right)\sum_{0\leq t\leq p-1}e^{\frac{i2\pi st}{p}}\left( \frac{t}{p}\right)\\[.3cm]
		&=\frac{p^{1/2}}{2}\nonumber,  
	\end{align} 
	where the absolute value of the quadratic symbol $|(s^{-1}\,|\,p)|=1$ for $s\ne0$.
\end{proof}

%CCCCCCCCCCCCCCCCCCCCCCCCCCCCCCCCCCCCCCCCCCCCCCCCCCCCCCCCCCCCCCCCCCCCCCCCCCCC
%CCCCCCCCCCCCCCCCCCCCCCCCCCCCCCCCCCCCCCCCCCCCCCCCCCCCCCCCCCCCCCCCCCCCCCCCCCCC
\begin{cor}   \label{cor11909FCQ.300A}\hypertarget{cor11909FCQ.300A} If $p$ is a large prime then the set $\mathscr{Q}_p=\{n^2:n\in [1,(p-1)/2]\}$ quadratic residues mod $p$ is a Salem set. In particular, the set of quadratic residues is equidistributed over the interval $[1,p-1]$.	
\end{cor} 
%SSSSSSSSSSSSSSSSSSSSSSSSSSSSSSSSSSSSSSSSSSSSSSSSSSSSSSSSSSSSSSSSSSSSSSSSSSSSS
%SSSSSSSSSSSSSSSSSSSSSSSSSSSSSSSSSSSSSSSSSSSSSSSSSSSSSSSSSSSSSSSSSSSSSSSSSSSSS
%SSSSSSSSSSSSSSSSSSSSSSSSSSSSSSSSSSSSSSSSSSSSSSSSSSSSSSSSSSSSSSSSSSSSSSSSSSSSS
%SSSSSSSSSSSSSSSSSSSSSSSSSSSSSSSSSSSSSSSSSSSSSSSSSSSSSSSSSSSSSSSSSSSSSSSSSSSSS
%SSSSSSSSSSSSSSSSSSSSSSSSSSSSSSSSSSSSSSSSSSSSSSSSSSSSSSSSSSSSSSSSSSSSSSSSSSSSS
%SSSSSSSSSSSSSSSSSSSSSSSSSSSSSSSSSSSSSSSSSSSSSSSSSSSSSSSSSSSSSSSSSSSSSSSSSSSSS
%SSSSSSSSSSSSSSSSSSSSSSSSSSSSSSSSSSSSSSSSSSSSSSSSSSSSSSSSSSSSSSSSSSSSSSSSSSSSS
\section{Finite Fourier Transforms of the Sets of Primitive Elements} \label{S1900FFT-PE}\hypertarget{S1900FFT-PE}
The structural properties of the finite Fourier transform of the characteristic functions of various sets of primitive elements are investigated in this section.
%PPPPPPPPPPPPPPPPPPPPPPPPPPPPPPPPPPPPPPPPPPPPPPPPPPPPPPPPPPPPPPPPPPPPPPPPPPPPP
%PPPPPPPPPPPPPPPPPPPPPPPPPPPPPPPPPPPPPPPPPPPPPPPPPPPPPPPPPPPPPPPPPPPPPPPPPPPPP
\begin{prop}   \label{prop1900FCP.300}\hypertarget{prop1900FCP.300} If $p$ is a large prime, then the finite Fourier transform $\widehat{\Psi(s)}$ of the characteristic function $\Psi(s)$ of primitive roots is basically a two values function
	\begin{equation}\label{eq1900FCP.300m}
		\widehat{\Psi(s)}=
		\begin{cases}
			\varphi(p-1) & \text{if } s=0,\\ 
			O(p^{1/2+\varepsilon})& \text{if } s\ne0,\\ 
		\end{cases}
	\end{equation}
	where $s\in\F_p$.
\end{prop} 
\begin{proof}[\textbf{Proof}] The case $s=0$ reduces to
	\begin{equation}\label{eq1900FCP.300p}
		\widehat{\Psi (0)}=\sum_{0\leq t\leq p-1} \Psi (t)=\varphi(p-1).
	\end{equation}
	Assume $s\ne0$ and start with the the \hyperlink{dfn5555FFT.300}{Definition} \ref{dfn5555FFT.300} of the finite Fourier transform $ \widehat{\Psi (s)}$ of the sequence of the characteristic function  $ \Psi (s)$ of primitive roots. Substituting the standard characteristic function, \hyperlink{lem2727CFFR.100-F}{Lemma} \ref{lem2727CFFR.100-F}, and partitioning it yield
	\begin{eqnarray}\label{eq1900FCP.300r}
		\widehat{\Psi (s)}&=&\sum_{0\leq t\leq p-1} \Psi (t)e^{ \frac{i2\pi s\cdot t}{p}}\\[.3cm]		&=&\sum_{0\leq t\leq p-1}\left( \frac{\varphi(p-1)}{p-1}\sum _{d\mid p-1} \frac{\mu(d)}{\varphi(d)}\sum _{\ord \chi =d} \chi(t)\right) e^{ \frac{i2\pi s\cdot t}{p}}\nonumber\\[.3cm]
		&=& \frac{\varphi(p-1)}{p-1}\sum_{0\leq t\leq p-1}e^{ \frac{i2\pi s\cdot t}{p}}+\frac{\varphi(p-1)}{p-1}\sum_{0\leq t\leq p-1,}\sum _{1<d\mid p-1} \frac{\mu(d)}{\varphi(d)}\sum _{\ord \chi =d} \chi(t)e^{ \frac{i2\pi s\cdot t}{p}}\nonumber\\[.3cm]
		&=&\frac{\varphi(p-1)}{p-1}\sum_{1\leq t\leq p-1,}\sum _{1<d\mid p-1} \frac{\mu(d)}{\varphi(d)}\sum _{\ord \chi =d} \chi(t)e^{ \frac{i2\pi s\cdot t}{p}},\nonumber
	\end{eqnarray}
	where $\chi(0)=0$, $\Psi (0)=\Psi (1)=0$ and $\sum_{0\leq t\leq p-1}e^{i2\pi st/p}=0$. Now observe that the Gauss sum
	\begin{equation}\label{eq1900FCP.300ms}
		\Bigg|\sum_{1\leq t\leq p-1} \chi(t)e^{ \frac{i2\pi s\cdot t}{p}}\Bigg|\leq p^{1/2}
	\end{equation}
	and the finite sum
	\begin{equation}\label{eq1900FCP.300t}
		\Bigg|\sum _{1<d\mid p-1} \frac{\mu(d)}{\varphi(d)}\sum _{\ord \chi =d}1\Bigg|\leq \sum _{1<d\mid p-1} |\mu(d)|\ll p^{\varepsilon},
	\end{equation}
	where $\varphi(d)$ is the number of characters $\chi\ne1$ of order $\ord \chi =d$ and $\varepsilon>0$ is a small number. Thus, switching the order of summation in the inner sum and taking absolute value yield
	\begin{eqnarray}\label{eq1900FCP.300u}
		\widehat{\Psi (s)}&=&\frac{\varphi(p-1)}{p-1}\sum_{1\leq t\leq p-1,}\sum _{1<d\mid p-1} \frac{\mu(d)}{\varphi(d)}\sum _{\ord \chi =d} \chi(t)e^{ \frac{i2\pi s\cdot t}{p}}\\[.3cm]
		&=&\frac{\varphi(p-1)}{p-1}\sum _{1<d\mid p-1} \frac{\mu(d)}{\varphi(d)}\sum _{\ord \chi =d,} \sum_{1\leq t\leq p-1} \chi(t)e^{ \frac{i2\pi s\cdot t}{p}}\nonumber\\[.3cm]		
		&\leq &\frac{\varphi(p-1)}{p-1}\left|\sum _{1<d\mid p-1} \frac{\mu(d)}{\varphi(d)}\sum _{\ord \chi =d}1\right|\cdot  \left|\sum_{1\leq t\leq p-1} \chi(t)e^{ \frac{i2\pi s\cdot t}{p}}\right|\nonumber\\[.3cm]
		&\ll& p^{1/2+\varepsilon}	\nonumber.	
	\end{eqnarray}
	This completes the verification.
\end{proof}

%CCCCCCCCCCCCCCCCCCCCCCCCCCCCCCCCCCCCCCCCCCCCCCCCCCCCCCCCCCCCCCCCCCCCCCCCCCCCC
%CCCCCCCCCCCCCCCCCCCCCCCCCCCCCCCCCCCCCCCCCCCCCCCCCCCCCCCCCCCCCCCCCCCCCCCCCCCCC
\begin{cor}   \label{cor11909SPC.300E}\hypertarget{cor11909SPC.300E} If $p$ is a large prime and $\tau\ne\pm 1,v^2$ is a primitive root, then the set $\mathscr{P}_p=\{\tau^n:\gcd(n,p-1)=1\}$ of primitive roots mod $p$ is a Salem set. In particular, the set of primitive roots is strongly equidistributed over the interval $[1,p-1]$.	
\end{cor} 

\begin{proof}[\textbf{Proof}] The upper bound $\widehat{\Psi (s)}\ll p^{1/2+\varepsilon}$ for $s\ne0$ satisfies the \hyperlink{dfn11900SP.200}{Definition} \ref{dfn11900SP.200} for a Salem set. To verify the strong equidistribution, observe that the Weyl criterion demands the optimal vanishing limit
	\begin{align}\label{eq11909SPC.300E2}
		\lim_{x\to\infty}\frac{1}{p}\sum_{0\leq t\leq p-1} \Psi (t)
		e^{ \frac{i2\pi s\cdot t}{p}	}=\lim_{x\to\infty}\frac{\widehat{\Psi (s)}}{p}
		=\lim_{x\to\infty}\frac{1}{p}\cdot p^{1/2+\varepsilon}=0
	\end{align}	
	foa any $s\ne0$. The strong equidistribution property follows from the fast convergence to 0, it vanishes at a rate close to $1/p^{1/2}$ as $p\to\infty$. 
\end{proof}

%SSSSSSSSSSSSSSSSSSSSSSSSSSSSSSSSSSSSSSSSSSSSSSSSSSSSSSSSSSSSSSSSSSSSSSSSSSSS
%SSSSSSSSSSSSSSSSSSSSSSSSSSSSSSSSSSSSSSSSSSSSSSSSSSSSSSSSSSSSSSSSSSSSSSSSSSSS
%SSSSSSSSSSSSSSSSSSSSSSSSSSSSSSSSSSSSSSSSSSSSSSSSSSSSSSSSSSSSSSSSSSSSSSSSSSSS
%SSSSSSSSSSSSSSSSSSSSSSSSSSSSSSSSSSSSSSSSSSSSSSSSSSSSSSSSSSSSSSSSSSSSSSSSSSSS
%SSSSSSSSSSSSSSSSSSSSSSSSSSSSSSSSSSSSSSSSSSSSSSSSSSSSSSSSSSSSSSSSSSSSSSSSSSSS
%SSSSSSSSSSSSSSSSSSSSSSSSSSSSSSSSSSSSSSSSSSSSSSSSSSSSSSSSSSSSSSSSSSSSSSSSSSSS
%SSSSSSSSSSSSSSSSSSSSSSSSSSSSSSSSSSSSSSSSSSSSSSSSSSSSSSSSSSSSSSSSSSSSSSSSSSSS
\section{Preliminary Evaluations and Estimates} \label{S47171SNEFR-C}\hypertarget{S47171SNEFR-C}
The disjoint union decomposition of the finite field $\F_{q^n}=\mathcal{A}\cup\mathcal{B}$, where
$\mathcal{A}$ is the subset of primitive normal elements and $\mathcal{B}$ is the subset of non primitive normal elements, will be used in the evaluations of several finite sums associated with the FFT of the characteristic function of primitive normal elements. The proof of \hyperlink{thm1900FCPN.600}{Theorem} \ref{thm1900FCPN.600} is broken up into four subsums. The estimates and evaluations of these four subsums are provided in  
\hyperlink{lem7171SNEFR.450-CaseI} {Lemma} \ref{lem7171SNEFR.450-CaseI} to \hyperlink{lem7171SNEFR.450-CaseIV} {Lemma} \ref{lem7171SNEFR.450-CaseIV} .
%LLLLLLLLLLLLLLLLLLLLLLLLLLLLLLLLLLLLLLLLLLLLLLLLLLLLLLLLLLLLLLLLLLLLLLLLLLLLLLLLLLLL
%LLLLLLLLLLLLLLLLLLLLLLLLLLLLLLLLLLLLLLLLLLLLLLLLLLLLLLLLLLLLLLLLLLLLLLLLLLLLLLLLLLLL
\begin{lem} \label{lem7171SNEFR.450-CaseI}\hypertarget{lem7171SNEFR.450-CaseI} Let $\tau\in \F_{q^n}$ be a fixed primitive normal element and let $\beta \in \F_{q^n}^{\times}$. If $t_1=0$ and $t_2=0$, then
	\begin{eqnarray}\label{eq7171SNEFR.450-CaseI2}
		T_{00}(\beta)
		&=&\sum_{\alpha\in \F_{q^n}}e^{\frac{i2\pi Tr(\alpha \beta)}{p}}\sum _{\substack{1\leq s\leq q^n-1\\\gcd (s,q^n-1)=1}} \frac{1}{q^n}\sum _{0\leq t_1\leq q^n-1} e^{\frac{i 2\pi(s-\log_{\tau}\alpha)t_1}{q^n}}\\[.3cm]
		&&\hskip 1.5 in \times\sum _{\substack{\deg s(x)\leq n-1\\\gcd (s(x),x^n-1)=1}} \frac{1}{q^n}\sum _{0\leq t_2\leq q^n-1} e^{ \frac{i2\pi(\log_{\tau}s(x)\circ\eta-\log_{\tau}\alpha)t_2}{q^n}}\nonumber\\[.3cm]		
		&=&0\nonumber	 .
	\end{eqnarray}	
\end{lem}

\begin{proof}[\textbf{Proof}] Substitute the parameters $t_1=0$ and $t_2=0$ in \eqref{eq7171SNEFR.450-CaseI2} and assume that the element $\alpha\in \F_{q^n}$ is not a primitive normal element. These steps yield
	\begin{eqnarray}\label{7171SNEFR.450-CaseI4}
		T_{00}(\beta)
		&=&\sum_{\alpha\in \F_{q^n}}e^{\frac{i2\pi Tr(\alpha \beta)}{p}}\sum _{\substack{1\leq s\leq q^n-1\\\gcd (s,q^n-1)=1}} \frac{1}{q^n} \;\times\;\sum _{\substack{\deg s(x)\leq n-1\\\gcd (s(x),x^n-1)=1}} \frac{1}{q^n}\\[.3cm]
		&=&\frac{\varphi (q^n-1)}{q^{n}}\cdot \frac{\Phi (x^n-1)}{q^{n}}\cdot \sum_{\alpha\in \F_{q^n}}e^{\frac{i2\pi Tr(\alpha \beta)}{p}}\nonumber\\[.3cm]		
		&=&0\nonumber	 .
	\end{eqnarray}	
\end{proof}

%LLLLLLLLLLLLLLLLLLLLLLLLLLLLLLLLLLLLLLLLLLLLLLLLLLLLLLLLLLLLLLLLLLLLLLLLLLLL
%LLLLLLLLLLLLLLLLLLLLLLLLLLLLLLLLLLLLLLLLLLLLLLLLLLLLLLLLLLLLLLLLLLLLLLLLLLLL
\begin{lem} \label{lem7171SNEFR.450-CaseII}\hypertarget{lem7171SNEFR.450-CaseII} Let $\tau\in \F_{q^n}$ be a fixed primitive normal element. If $t_1\in [0,q^n-1]$ and $t_2=0$, then
	\begin{eqnarray}\label{eq7171SNEFR.450-CaseII2}
		T_{01}(\beta)
		&=&\sum_{\alpha\in \F_{q^n}}e^{\frac{i2\pi Tr(\alpha \beta)}{p}}\sum _{\substack{1\leq s\leq q^n-1\\\gcd (s,q^n-1)=1}} \frac{1}{q^n}\sum _{0\leq t_1\leq q^n-1} e^{\frac{i 2\pi(s-\log_{\tau}\alpha)t_1}{q^n}}\\[.3cm]
		&&\hskip 1.5 in \times\sum _{\substack{\deg s(x)\leq n-1\\\gcd (s(x),x^n-1)=1}} \frac{1}{q^n}\sum _{0\leq t_2\leq q^n-1} e^{ \frac{i2\pi(\log_{\tau}s(x)\circ\eta-\log_{\tau}\alpha)t_2}{q^n}}\nonumber\\[.3cm]
		&\ll &q^{n/2}\nonumber.
	\end{eqnarray}	
\end{lem}
\begin{proof}[\textbf{Proof}] The finite sum is broken up into two cases. \\

\textbf{Case I:} The element $\alpha\in \mathcal{A}$ is a primitive normal element. Substitute the parameters $t_1\in [0,q^n-1]$ and $t_2=0$ in \eqref{eq7171SNEFR.450-CaseII2}. These steps yield

\begin{align}\label{eq7171PNEFR.450II5}
	T_{01A}(\beta)
	&=\sum_{\alpha\in \mathcal{A}}e^{\frac{i2\pi Tr(\alpha \beta)}{p}}\sum _{\substack{1\leq s\leq q^n-1\\\gcd (s,q^n-1)=1}} \frac{1}{q^n}\sum _{0\leq t_1\leq q^n-1} e^{\frac{i 2\pi(s-\log_{\tau}\alpha)t_1}{q^n}}\; \times\;\sum _{\substack{\deg s(x)\leq n-1\\\gcd (s(x),x^n-1)=1}} \frac{1}{q^n}\nonumber\\[.3cm]
	&=\sum_{\alpha\in \mathcal{A}}e^{\frac{i2\pi Tr(\alpha \beta)}{p}}\;\times 1\; \times\;\sum _{\substack{\deg s(x)\leq n-1\\\gcd (s(x),x^n-1)=1}} \frac{1}{q^n}\nonumber\\[.3cm]
	&= \frac{\Phi (x^n-1)}{q^n}\sum_{\alpha\in \mathcal{A}}e^{\frac{i2\pi Tr(\alpha \beta)}{p}}\nonumber\\[.3cm]
	&\ll q^{n/2},
\end{align}	
since the condition that $\alpha\in \F_{q^n}$ is a primitive normal element implies that $i 2\pi(s-\log_{\tau}\alpha)=0$ for some $s\in [1,q^n-1]$ such that $\gcd (s,q^n-1)=1$ and   the estimate for the exponential sum follows from \hyperlink{lem2727CEES.650U}{Lemma} \ref{lem2727CEES.650U}. 	\\

\textbf{Case II:} The element $\alpha\in \mathcal{B}\subset \F_{q^n}$ is not a primitive normal element.	Substitute the parameters $t_1\in [0,q^n-1]$ and $t_2=0$ in \eqref{eq7171SNEFR.450-CaseII2}. These steps yield
	
	\begin{align}\label{eq7171PNEFR.450II7}
		T_{01B}(\beta)
		&=\sum_{\alpha\in \mathcal{B}}e^{\frac{i2\pi Tr(\alpha \beta)}{p}}\sum _{\substack{1\leq s\leq q^n-1\\\gcd (s,q^n-1)=1}} \frac{1}{q^n}\sum _{0\leq t_1\leq q^n-1} e^{\frac{i 2\pi(s-\log_{\tau}\alpha)t_1}{q^n}}\; \times\;\sum _{\substack{\deg s(x)\leq n-1\\\gcd (s(x),x^n-1)=1}} \frac{1}{q^n}\nonumber\\[.3cm]
		&=\sum_{\alpha\in \mathcal{B}}e^{\frac{i2\pi Tr(\alpha \beta)}{p}}\sum _{\substack{1\leq s\leq q^n-1\\\gcd (s,q^n-1)=1}} \frac{(-1)}{q^n}\; \times\;\sum _{\substack{\deg s(x)\leq n-1\\\gcd (s(x),x^n-1)=1}} \frac{1}{q^n}\nonumber\\[.3cm]
	&=-\frac{\varphi (q^n-1)}{q^{n}}\cdot \frac{\Phi (x^n-1)}{q^{n}}\cdot \sum_{\alpha\in \mathcal{B}}e^{\frac{i2\pi Tr(\alpha \beta)}{p}},
	\end{align}	
	since the condition that $\alpha\in \mathcal{B}$ is a non primitive normal element implies that $i 2\pi(s-\log_{\tau}\alpha)\ne0$ for all $s\in [1,q^n-1]$ such that $\gcd (s,q^n-1)=1$.  Next, the disjoint decomposition $\F_{q^n}=\mathcal{A}\cup\mathcal{B}$ implies that
	\begin{align}\label{eq7171PNEFR.450II9}
\sum_{\alpha\in \F_{q^n}}e^{\frac{i2\pi Tr(\alpha \beta)}{p}}
	&=\sum_{\alpha\in \mathcal{A}}e^{\frac{i2\pi Tr(\alpha \beta)}{p}}+\sum_{\alpha\in \mathcal{B}}e^{\frac{i2\pi Tr(\alpha \beta)}{p}}=0.
\end{align}		
Lastly, applying \hyperlink{lem2727CEES.650U}{Lemma} \ref{lem2727CEES.650U} to the exponential sum yields
	\begin{align}\label{eq7171PNEFR.450II12}
	T_{01}(\beta)
		&=\frac{\varphi (q^n-1)}{q^{n}}\cdot \frac{\Phi (x^n-1)}{q^{n}}\cdot \sum_{\alpha\in \mathcal{A}}e^{\frac{i2\pi Tr(\alpha \beta)}{p}}\nonumber\\[.3cm]	
	&\ll\frac{\varphi (q^n-1)}{q^{n}}\cdot \frac{\Phi (x^n-1)}{q^{n}}\cdot q^{n/2} \\
	&\ll q^{n/2}\nonumber,
\end{align}	
where $\varphi(q^n-1)/q^{n}\leq 1$ and $\Phi(x^n-1)/q^{n}\leq 1$. Summing the subsums \begin{equation}
	T_{01}(\beta)=T_{01A}(\beta)+T_{01B}(\beta)\ll q^{n/2}
\end{equation} completes the verification of the upper bound.
\end{proof}
%LLLLLLLLLLLLLLLLLLLLLLLLLLLLLLLLLLLLLLLLLLLLLLLLLLLLLLLLLLLLLLLLLLLLLLLLLLLLLLLLLLLL
%LLLLLLLLLLLLLLLLLLLLLLLLLLLLLLLLLLLLLLLLLLLLLLLLLLLLLLLLLLLLLLLLLLLLLLLLLLLLLLLLLLLL
\begin{lem} \label{lem7171SNEFR.450-CaseIII}\hypertarget{lem7171SNEFR.450-CaseIII} Let $\tau\in \F_{q^n}$ be a fixed primitive normal element. If $t_1\in [0,q^n-1]$ and $t_2=0$, then
	\begin{eqnarray}\label{eq7171SNEFR.450-CaseIII2}
		T_{10}(\beta)
		&=&\sum_{\alpha\in \F_{q^n}}e^{\frac{i2\pi Tr(\alpha \beta)}{p}}\sum _{\substack{1\leq s\leq q^n-1\\\gcd (s,q^n-1)=1}} \frac{1}{q^n}\sum _{0\leq t_1\leq q^n-1} e^{\frac{i 2\pi(s-\log_{\tau}\alpha)t_1}{q^n}}\\[.3cm]
		&&\hskip 1.5 in \times\sum _{\substack{\deg s(x)\leq n-1\\\gcd (s(x),x^n-1)=1}} \frac{1}{q^n}\sum _{0\leq t_2\leq q^n-1} e^{i2\pi \frac{\tr(s(x)\circ\tau-\alpha)t_2}{q^n}}\nonumber\\[.3cm]
			&\ll& q^{n/2}\nonumber.
	\end{eqnarray}	
\end{lem}
\begin{proof}[\textbf{Proof}] Let $\F_{q^n}=\mathcal{A}\cup\mathcal{B}$, where
	$\mathcal{A}$ is the subset of primitive normal elements and $\mathcal{B}$ is the subset of non primitive normal elements be a disjoint decomposition. Break it up into two subsums $T_{10A}(\beta)$ and $T_{10B}(\beta)$ as in the previous result and make the necessary changes to complete the verification. 
\end{proof}

%LLLLLLLLLLLLLLLLLLLLLLLLLLLLLLLLLLLLLLLLLLLLLLLLLLLLLLLLLLLLLLLLLLLLLLLLLLLL
%LLLLLLLLLLLLLLLLLLLLLLLLLLLLLLLLLLLLLLLLLLLLLLLLLLLLLLLLLLLLLLLLLLLLLLLLLLLL
\begin{lem} \label{lem7171SNEFR.450-CaseIV}\hypertarget{lem7171SNEFR.450-CaseIV} Let $\tau\in \F_{q^n}$ be a fixed primitive normal element. If the element $\alpha$ is not a primitive normal element and $t_1\ne0$ and $t_2\ne0$, then
	\begin{eqnarray}\label{eq7171SNEFR.450-CaseIV40}
		T_{11}(\beta)
		&=&\sum_{\alpha\in \F_{q^n}}e^{\frac{i2\pi Tr(\alpha \beta)}{p}}\sum _{\substack{1\leq s\leq q^n-1\\\gcd (s,q^n-1)=1}} \frac{1}{q^n}\sum _{0\leq t_1\leq q^n-1} e^{\frac{i 2\pi(s-\log_{\tau}\alpha)t_1}{q^n}}\\[.3cm]
		&&\hskip 1.5 in \times\sum _{\substack{\deg s(x)\leq n-1\\\gcd (s(x),x^n-1)=1}} \frac{1}{q^n}\sum _{0\leq t_2\leq q^n-1} e^{ \frac{i2\pi(\log_{\tau}s(x)\circ\eta-\log_{\tau}\alpha)t_2}{q^n}}\nonumber\\[.3cm]
		&\ll&q^{n/2}\nonumber.
	\end{eqnarray}	 
\end{lem}

\begin{proof}[\textbf{Proof}] Consider the two subsums $T_{11A}(\beta)$ and $T_{11B}(\beta)$ specified by the dyadic decomposition $\F_{q^n}=\mathcal{A}\cup\mathcal{B}$. \\
	
\textbf{Case I:} The element $\alpha\in \mathcal{A}\subset \F_{q^n}$ is a primitive normal element. The first subsum reduces to	
	\begin{eqnarray}\label{eq7171SNEFR.450-CaseIV42}
		T_{11A}(\beta)
		&=&\sum_{\alpha\in \mathcal{A}}e^{\frac{i2\pi Tr(\alpha \beta)}{p}}\sum _{\substack{1\leq s\leq q^n-1\\\gcd (s,q^n-1)=1}} \frac{1}{q^n}\sum _{1\leq t_1\leq q^n-1} e^{\frac{i 2\pi(s-\log_{\tau}\alpha)t_1}{q^n}}\\[.3cm]
		&&\hskip 1.5 in \times\sum _{\substack{\deg s(x)\leq n-1\\\gcd (s(x),x^n-1)=1}} \frac{1}{q^n}\sum _{1\leq t_2\leq q^n-1} e^{ \frac{i2\pi(\log_{\tau}s(x)\circ\eta-\log_{\tau}\alpha)t_2}{q^n}}\nonumber\\[.3cm]
		&=&\sum_{\alpha\in \mathcal{A}}e^{\frac{i2\pi Tr(\alpha \beta)}{p}}\left(\frac{q^n-1}{q^n} \right) \times\left(\frac{q^n-1}{q^n} \right)\nonumber\\[.3cm]	
		&\leq&\sum_{\alpha\in \mathcal{A}}e^{\frac{i2\pi Tr(\alpha \beta)}{p}}\nonumber\\[.3cm]
		&\ll&q^{n/2}\nonumber.
	\end{eqnarray}
The last line in  \eqref{eq7171SNEFR.450-CaseIV42} follows from \hyperlink{lem2727CEES.650U}{Lemma} \ref{lem2727CEES.650U}.\\

\textbf{Case II:} The element $\alpha\in \mathcal{B}\subset \F_{q^n}$ is not a primitive normal element. The second subsum reduces to	
	\begin{eqnarray}\label{eq7171SNEFR.450-CaseIV44}
	T_{11B}(\beta)
	&=&\sum_{\alpha\in \mathcal{B}}e^{\frac{i2\pi Tr(\alpha \beta)}{p}}\sum _{\substack{1\leq s\leq q^n-1\\\gcd (s,q^n-1)=1}} \frac{1}{q^n}\sum _{1\leq t_1\leq q^n-1} e^{\frac{i 2\pi(s-\log_{\tau}\alpha)t_1}{q^n}}\\[.3cm]
	&&\hskip 1.5 in \times\sum _{\substack{\deg s(x)\leq n-1\\\gcd (s(x),x^n-1)=1}} \frac{1}{q^n}\sum _{1\leq t_2\leq q^n-1} e^{ \frac{i2\pi(\log_{\tau}s(x)\circ\eta-\log_{\tau}\alpha)t_2}{q^n}}\nonumber\\[.3cm]
	&=&\sum_{\alpha\in \mathcal{B}}e^{\frac{i2\pi Tr(\alpha \beta)}{p}}\left(-1 \right) \times\left(-1 \right)\nonumber.
\end{eqnarray}
Next, the disjoint decomposition $\F_{q^n}=\mathcal{A}\cup\mathcal{B}$ implies that
\begin{align}\label{eq7171SNEFR.450-CaseIV46}
	\sum_{\alpha\in \F_{q^n}}e^{\frac{i2\pi Tr(\alpha \beta)}{p}}
	&=\sum_{\alpha\in \mathcal{A}}e^{\frac{i2\pi Tr(\alpha \beta)}{p}}+\sum_{\alpha\in \mathcal{B}}e^{\frac{i2\pi Tr(\alpha \beta)}{p}}=0.
\end{align}	
Hence, substituting this identity and using \hyperlink{lem2727CEES.650U}{Lemma} \ref{lem2727CEES.650U} yield
	\begin{eqnarray}\label{eq7171SNEFR.450-CaseIV48}
	T_{11B}(\beta)
	&=&\sum_{\alpha\in \mathcal{B}}e^{\frac{i2\pi Tr(\alpha \beta)}{p}}\\[.3cm]
	&=&-\sum_{\alpha\in \mathcal{A}}e^{\frac{i2\pi Tr(\alpha \beta)}{p}}\nonumber\\[.3cm]
	&\ll&q^{n/2}\nonumber.
\end{eqnarray}
Summing the subsums \begin{equation}
	T_{11}(\beta)=T_{11A}(\beta)+T_{11B}(\beta)\ll q^{n/2}
\end{equation} completes the verification of the upper bound.
\end{proof}

%SSSSSSSSSSSSSSSSSSSSSSSSSSSSSSSSSSSSSSSSSSSSSSSSSSSSSSSSSSSSSSSSSSSSSSSSSSSS
%SSSSSSSSSSSSSSSSSSSSSSSSSSSSSSSSSSSSSSSSSSSSSSSSSSSSSSSSSSSSSSSSSSSSSSSSSSSS
%SSSSSSSSSSSSSSSSSSSSSSSSSSSSSSSSSSSSSSSSSSSSSSSSSSSSSSSSSSSSSSSSSSSSSSSSSSSS
%SSSSSSSSSSSSSSSSSSSSSSSSSSSSSSSSSSSSSSSSSSSSSSSSSSSSSSSSSSSSSSSSSSSSSSSSSSSS
%SSSSSSSSSSSSSSSSSSSSSSSSSSSSSSSSSSSSSSSSSSSSSSSSSSSSSSSSSSSSSSSSSSSSSSSSSSSS
%SSSSSSSSSSSSSSSSSSSSSSSSSSSSSSSSSSSSSSSSSSSSSSSSSSSSSSSSSSSSSSSSSSSSSSSSSSSS
%SSSSSSSSSSSSSSSSSSSSSSSSSSSSSSSSSSSSSSSSSSSSSSSSSSSSSSSSSSSSSSSSSSSSSSSSSSSS
\section{Finite Fourier Transform of the Set of Primitive Normal Elements} \label{S1900FFT-PN}\hypertarget{S1900FFT-PN}
The structural properties of the finite Fourier transform of the characteristic function of primitive normal elements in finite fields is investigated in this section. 
%PPPPPPPPPPPPPPPPPPPPPPPPPPPPPPPPPPPPPPPPPPPPPPPPPPPPPPPPPPPPPPPPPPPPPPPPPPPPPPPPP

%PPPPPPPPPPPPPPPPPPPPPPPPPPPPPPPPPPPPPPPPPPPPPPPPPPPPPPPPPPPPPPPPPPPPPPPPPPPPPPPPP
\begin{thm}   \label{thm1900FCPN.600}\hypertarget{thm1900FCPN.600} Let $q$ be a prime power and let $n\geq2$. Then the finite Fourier transform $\widehat{\mathcal{C}(\beta)}$ of the characteristic function $\mathcal{C}(\alpha)$ of of primitive normal elements in $\F_{q^n}$ is approximately a two values function
	\begin{equation}\label{eq1900FCQ.600m}
		\widehat{\mathcal{C}(\beta)}=
		\begin{cases}
			\frac{\varphi (q^n-1)}{q^n}\frac{\Phi (x^n-1)}{q^n}(1+o(1)) q^n& \text{if } \beta=0,\\[.3cm] 
			O(q^{n/2})& \text{if } \beta\ne0,\\ 
		\end{cases}
	\end{equation}
	where $\beta\in \F_{q^n}$. 
\end{thm} 

\begin{proof}[\textbf{Proof}] By \hyperlink{dfn5555FFT.300}{Definition} \ref{dfn5555FFT.300}, the finite Fourier transform $\widehat{\mathcal{C}}:\F_{q^n}\longrightarrow \C$ of the characteristic function $\mathcal{C}:\F_{q^n}\longrightarrow \{0,1\}$ of primitive normal elements, see \hyperlink{lem7171CFPE.150PDF}{Lemma} \ref{lem7171CFPE.150PDF} and \hyperlink{lem7171CFNE.150NDF}{Lemma} \ref{lem7171CFNE.150NDF}, is the followings:
	\begin{equation}
		\widehat{\mathcal{C} (\beta)}
		=\sum_{\alpha\in \F_{q^n}}\mathcal{C}(\alpha)e^{\frac{i2\pi Tr(\alpha \beta)}{p}}=\sum_{\alpha\in \F_{q^n}}e^{\frac{i2\pi Tr(\alpha \beta)}{p}}	\Psi (\alpha)	\Psi_q (\alpha),
	\end{equation}
where $Tr:\F_{q^n}\longrightarrow\C$ is the absolute trace function and $\beta\in \F_{q^n}$. The case $\beta=0$ reduces to
	\begin{equation}
	\widehat{\mathcal{C} (0)}
	=\sum_{\alpha\in \F_{q^n}}\mathcal{C}(\alpha)e^{\frac{i2\pi Tr(\alpha \beta)}{p}}=\sum_{\alpha\in \F_{q^n}}\Psi (\alpha)	\Psi_q (\alpha)=\frac{\varphi (q^n-1)}{q^{n}}\cdot\frac{\Phi (x^n-1)}{q^{n}}\cdot q^n.
\end{equation}

This coefficient of the FFT corresponds to the main term of the counting function of primitive normal elements. \\

Assume $\beta\ne0$. Replacing the formulas of these indicator functions, given in \hyperlink{lem7171CFPE.150PDF}{Lemma} \ref{lem7171CFPE.150PDF} and  \hyperlink{lem7171CFNE.150NDF}{Lemma} \ref{lem7171CFNE.150NDF}, respectively, yields
	\begin{align}\label{eq1900FCQ.600e}
		\widehat{\mathcal{C} (\beta)}
		&=\sum_{\alpha\in \F_{q^n}}e^{\frac{i2\pi Tr(\alpha \beta)}{p}}	\Psi (\alpha)	\Psi_q (\alpha)\\[.2cm]
		&=\sum_{\alpha\in \F_{q^n}}e^{\frac{i2\pi Tr(\alpha \beta)}{p}}\sum _{\substack{1\leq s\leq q^n-1\\\gcd (s,q^n-1)=1}} \frac{1}{q^n}\sum _{0\leq t_1\leq q^n-1} e^{\frac{i 2\pi(s-\log_{\tau}\alpha)t}{q^n}}\nonumber\\[.3cm]	
		&\hskip 1.75 in \times \sum _{\substack{\deg s(x)\leq n-1\\\gcd (s(x),x^n-1)=1}} \frac{1}{q^n}\sum _{0\leq t_2\leq q^n-1} e^{ \frac{i2\pi(\log_{\tau}s(x)\circ\eta-\log_{\tau}\alpha)t}{q^n}}\nonumber\\[.3cm]	
		&=T_{00}(\beta)\;+\;T_{01}(\beta)\;+\;T_{10}(\beta)\;+\;T_{11}(\beta)\nonumber.
	\end{align}
	The four subsums $T_{ij}(\beta)$ are determined by four proper subsets 
	\begin{equation}\label{eq1900FCQ.600e5}
		\mathcal{T}_{ij}=\{(t_1,t_2)\}\subset[0,q^n-1]\times [0,q^n-1].
	\end{equation}
	These subsums are evaluated and estimated independently in \hyperlink{lem7171SNEFR.450-CaseI} {Lemma} \ref{lem7171SNEFR.450-CaseI} to \hyperlink{lem7171SNEFR.450-CaseIV} {Lemma} \ref{lem7171SNEFR.450-CaseIV}. Substituting these evaluations and estimates yield
	\begin{eqnarray}\label{eq1900FCQ.600k11}
		\widehat{\mathcal{C} (\beta)}
		&=&T_{00}(\beta)\;+\;T_{01}(\beta)\;+\;T_{10}(\beta)\;+\;T_{11}(\beta)\\[.3cm]
		&=&0+q^{n/2}+q^{n/2}+q^{n/2}\nonumber\\[.3cm]
		&=&O\left( q^{n/2}\right)\nonumber.
	\end{eqnarray}	
This completes the proof.
\end{proof}

%CCCCCCCCCCCCCCCCCCCCCCCCCCCCCCCCCCCCCCCCCCCCCCCCCCCCCCCCCCCCCCCCCCCCCCCCCCCCCCCC
%CCCCCCCCCCCCCCCCCCCCCCCCCCCCCCCCCCCCCCCCCCCCCCCCCCCCCCCCCCCCCCCCCCCCCCCCCCCCCCCC
The proof of \hyperlink{thm4343PNEFF.050}{Theorem} \ref{thm4343PNEFF.050} is part of corollary stated below. 
\begin{cor}   \label{cor11909FCQ.350}\hypertarget{or11909FCQ.350} If $q$ is a large prime power then the set 
	\begin{equation}
		\mathcal{PN}(q)=\{\alpha\in \F_{q^n} \text{ is a primitive normal element}\}
	\end{equation} is a Salem set. In particular, the primitive normal elements are equidistributed in $\F_{q^n}$.	
\end{cor} 

\begin{proof}[\textbf{Proof}] The upper bound $\widehat{\mathcal{C} (\beta)}\ll q^{n/2}$ for $\beta\ne0$ satisfies the \hyperlink{dfn11900SP.200}{Definition} \ref{dfn11900SP.200} for a Salem set. To verify the strong equidistribution, observe that the Weyl criterion demands the optimal vanishing limit
	\begin{align}\label{eq11909SPC.350E2}
		\lim_{q^n\to\infty}\frac{1}{q^n}\sum_{\alpha\in\F_{q^n}} \mathcal{C} (\alpha)
		e^{ \frac{i2\pi Tr(\alpha \beta)}{p}	}=\lim_{q^n\to\infty}\frac{\widehat{\mathcal{C} (\beta)}}{q^n}
		=\lim_{q^n\to\infty}\frac{1}{q^n}\cdot q^{n/2}=0
	\end{align}	
	foa any $\beta\ne0$. The strong equidistribution property follows from the fast convergence to 0, it vanishes at a rate close to $1/q^{n/2}$ as $q^n\to\infty$. 
\end{proof}
%BBBBBBBBBBBBBBBBBBBBBBBBBBBBBBBBBBBBBBBBBBBBBBBBBBBBBBBBBBBBBBBB
%BBBBBBBBBBBBBBBBBBBBBBBBBBBBBBBBBBBBBBBBBBBBBBBBBBBBBBBBBBBBBBBB
%BBBBBBBBBBBBBBBBBBBBBBBBBBBBBBBBBBBBBBBBBBBBBBBBBBBBBBBBBBBBBBBB
%BBBBBBBBBBBBBBBBBBBBBBBBBBBBBBBBBBBBBBBBBBBBBBBBBBBBBBBBBBBBBBBB
%%%%%%%%%%%%%%%%%%%%%%%%%%%%%% Bibliography%%%%%%%%%%%%%%%%%%%%%%
%\newpage
{\small

\end{document}